\documentclass[10pt]{amsart}	
\usepackage{amsfonts}
\usepackage{amssymb}
\usepackage{amsmath}
\usepackage[arrow,matrix,curve]{xy}

\usepackage[active]{srcltx}

\usepackage{umlaute}

\usepackage[normalem]{ulem}
\usepackage{color}

\newcommand{\C}{{\mathbb C}}

\newcommand{\Z}{{\mathbb Z}}

\newcommand{\E}{{\mathbb E}}

\newtheorem{theorem}{Theorem}[section]
\newtheorem{lemma}[theorem]{Lemma}

\newtheorem{corollary}[theorem]{Corollary}
\newtheorem{proposition}[theorem]{Proposition}
\newtheorem{definition}[theorem]{Definition}

\def\cal{\mathcal}

\newcommand{\call}[0]{{\cal L}}

\newcommand{\cale}[0]{{\cal E}}
\newcommand{\calk}[0]{{\cal K}}

\newcommand{\calm}[0]{{\cal M}}

\begin{document}

\title[Inverse semigroup equivariant $KK$-theory]{Inverse semigroup equivariant $KK$-theory and $C^*$-extensions} 
\author[Burgstaller]{Bernhard Burgstaller}
\email{bernhardburgstaller@yahoo.de}
\subjclass{19K35, 20M18, 19K33}
\keywords{$KK$-theory, Ext, extension group, inverse semigroup}

\begin{abstract}
In this note we extend the classical result by G. G. Kasparov that the Kasparov groups $KK_1(A,B)$
can be identified with the extension groups $\mbox{Ext}(A,B)$ to the inverse semigroup equivariant setting.
More precisely, we show that $KK_G^1(A,B) \cong
\mbox{Ext}_G(A \otimes \calk_G,B \otimes \calk_G)$
for every countable, $E$-continuous inverse semigroup $G$. 
For locally compact second countable groups $G$
this was proved by K. Thomsen, and technically 
this note
presents an adaption of his proof.
\end{abstract}

\maketitle

%
%


\section{}

In Theorem 7.1 of \cite{kasparov1981}, G. G. Kasparov shows that for every 
{\em compact} second countable group $G$
and ungraded separable $G$-algebras $A$ and $B$ (where $A$ is nuclear) there exists an isomorphism
between the extension group $\mbox{Ext}(A,B)$ and the Kasparov group $KK^1(A,B):=KK(A,B \oplus B)$.
In \cite{thomsen}, K. Thomsen generalizes this result to locally compact second countable Hausdorff groups $G$
and ungraded separable $G$-algebras $A$ and $B$
by establishing 
an isomorphism between $\mbox{Ext}_G(A \otimes \calk_G,B \otimes \calk_G)$
and $KK^1_G(A,B)$, where
$\calk_G = \calk \otimes \calk(L^2(G))$. 
Very roughly speaking the proof is done by considering Kasparov's proof for these more general groups $G$.
Some unitaries which fall out due to $G$-equivariance 
during the process of equivalently transforming Kasparov cycles to another format
called $KK$-cocycles
are not
averaged away, but kept as unitary cocycles and then transferred into the algebra $\calk_G$
by some equivariance theorems by J. A. Mingo and W. J. Phillips \cite{mingophillips}.

In this note we prove the analogous result for countable, $E$-continuous inverse semigroups $G$
by adapting K. Thomsen's proof. 

We will here not repeat Thomsen's complete proof but only show the modifications that have to be made
when reading Thomsen's paper for an inverse semigroup $G$.
We will be mainly focused on details, leaving the greater context to the original paper by Thomsen.
The only thing we have to observe and to take care of is the $G$-action.
The definitions of equivariant $KK$-theory for groups and inverse semigroups are very close,
see Section \ref{section2},
but the difference that $g g^{-1}=1$ in a group but $g g^{-1}$ is just a projection
in an inverse semigroup, is unavoidable.
The short answer of this note is that we shall make only the following modifications
in Thomsen's paper \cite{thomsen}:

\begin{itemize}
\item
If one side of an identity 
contains an action $\alpha_{g}$ but the other side does not, then we have
to add $\alpha_{g g^{-1}}$ to the other side. 
For example, $\alpha_g(x)=y$ 
has to be changed to
$\alpha_g(x)=\alpha_{g g^{-1}}(y)$. 
This is obvious since $\alpha_g$ looses information by
projecting onto $\alpha_{g g^{-1}}$.

\item
A unitary cocycle 
has to be replaced by a cocycle in the sense of Definition \ref{defunitarycocycle}. 
It is obvious that we cannot work with unitaries.
The unitary operator definition changes to a partial isometry definition with identical source and range projection.

\item
The 
$\ell^2(G)$-space 
has to be completely replaced by a {compatible} $\ell^2(G)$-space,
see Definition \ref{defCompatibleL2}.
Without a suitable change the algebra $\calk(\ell^2(G))$ were no longer a $G$-algebra
and the proof would break down. The construction of this compatible $\ell^2(G)$-space requires
however the inverse semigroup $G$ to be {\em $E$-continuous}, see Definition \ref{definitionEcontinuous}.
\end{itemize}

More or less we could finish here our note, but we shall nevertheless discuss some
of the modifications in detail to convince the reader that everything goes through.

An important help, and a fact which is used again and again, is that the operator $\alpha_{g g^{-1}} \in Z \calm(A)$
is always in the center of the multiplier algebra of a given $G$-algebra $(A,\alpha)$. This includes also 
the $G$-algebra of adjoint-able operators $\call(\cale)$ of a given $G$-Hilbert module $\cale$.
In Thomsen's proof $\alpha_{g g^{-1}}=1$ and so $\alpha_{g g^{-1}}$ is absorbed by any neighboring element via multiplication. In our proof $\alpha_{g g^{-1}}$ often cannot vanish in its
position, but since it is in the center, can move around within an expression until it is absorbed elsewhere, 
for example another presence of $\alpha_g$ through the identity $\alpha_{g g^{-1}} \alpha_g = \alpha_g$.
With that technique, Thomsen's proof can be kept under 
the weaker assumptions.


A different approach in generalizing Kasparov's result to
locally compact groups is given by J. Cuntz \cite{cuntz1998}.

The structure of this note is as follows.
The Sections \ref{section2} to \ref{section7}
represent corresponding sections of Thomsen's paper,
with the same numbers and titles.
There we discuss in each section the modifications
that have to be performed
in the corresponding section in \cite{thomsen}.
Section \ref{section2} includes also a brief summary of inverse semigroup equivariant $KK$-theory,
and Section \ref{section7} contains the above mentioned main result in
Theorems \ref{theorem71} and \ref{theorem72}.
The last Section \ref{sectionMingoPhillips} presents an appendix where
we introduce the compatible $\ell^2(G)$-space and adapt certain triviality results
from Mingo and Phillips \cite{mingophillips} for this 
$\ell^2(G)$-space.



\section{The Busby-invariant in the equivariant case}

\label{section2}

Let $G$ denote a countable inverse semigroup. We shall denote
the involution
on $G$ by $g \mapsto g^{-1}$ (determined by $g g^{-1} g = g$).
A semigroup homomorphism is said to be {\em unital} if it preserves the identity $1 \in G$
and the zero element $0 \in G$, provided that the involved inverse semigroups are gifted with any of these elements.
In this note, we shall however require none of them. 
We consider $G$-equivariant $KK$-theory as defined in \cite{burgiSemimultiKK}
but make a slight adaption by making this theory
{\em compatible} in the sense that we only allow for {\em compatible} Hilbert (bi)modules (cf. Definition \ref{defHilbert}).
We are going to recall the basic definitions of $KK^G$.

\begin{definition}   \label{defCstar}
{\rm
A {\em $G$-algebra} $(A,\alpha)$ is a $\Z/2$-graded $C^*$-algebra $A$ with a
unital semigroup homomorphism
$\alpha: G \rightarrow \mbox{End}(A)$ such that
$\alpha_g$ respects the grading
and $\alpha_{g g^{-1}}(x) y = x \alpha_{g g^{-1}}(y)$
for all $x,y \in A$ and $g \in G$.
}
\end{definition}

\begin{definition}   \label{defHilbert}
{\rm
A {\em $G$-Hilbert $B$-module} $\cale$ is a $\Z/2$-graded Hilbert
module over a $G$-algebra $(B,\beta)$ endowed with a unital
semigroup homomorphism $G \rightarrow \mbox{Lin}(\cale)$ (linear maps on $\cale$)
such that $U_g$ respects the grading and
$\langle U_g(\xi),U_g(\eta)\rangle = \beta_g(\langle \xi,\eta \rangle)$,
$U_g(\xi b) = U_g(\xi) \beta_g(b)$,
and
$U_{g g^{-1}}(\xi) b = \xi \beta_{g g^{-1}}( b)$
for all $g \in G,\xi,\eta \in \cale$ and $b \in B$.
}
\end{definition}

In the last definition, $U_{g g^{-1}}$ is automatically a self-adjoint projection in the {\em center} of
$\call(\cale)$ (because adjoint-able operators are $B$-linear and $U_{g g^{-1}}(\xi) b = \xi \beta_{g g^{-1}}( b)$),
and the $G$-action $G \rightarrow \mbox{End}(\call(\cale))$ given by $g(T) =
U_g T U_{g^{-1}}$ turns $\call(\cale)$ to a $G$-algebra ($g \in G$ and $T \in \call(\cale)$).
A $G$-algebra $(A,\alpha)$
is a $G$-Hilbert module over itself under the inner product $\langle a,b\rangle
= a^* b$ and
$U: = \beta := \alpha$ in the last definition.
A $*$-homomorphism between $G$-algebras is called {\em $G$-equivariant} if it intertwines
the $G$-action.
A {\em $G$-Hilbert $A,B$-bimodule} over $G$-algebras $A$ and $B$
is a $G$-Hilbert $B$-module $\cale$ equipped with a $G$-equivariant $*$-homomorphism
$A \rightarrow \call(\cale)$.
The compact operators on a separable Hilbert space are written as $\calk$, $\calk(\cale) \subseteq \call(\cale)$ denotes the compact operators
on a Hilbert module $\cale$, and $\theta_{x,y} \in \calk(\cale)$ the elementary compact operators
$\theta_{x,y}(z):= x \langle y,z\rangle$ for all $x,y,z \in \cale$.

\begin{definition}  \label{defCycle}
{\rm
Let $A$ and $B$ be $G$-algebras.
We define a Kasparov cycle $(\cale,T)$, where $\cale$ is a $G$-Hilbert $A,B$-bimodule, to be an ordinary Kasparov cycle (without $G$-action) (see \cite{kasparov1981,kasparov1988}) satisfying $U_g T U_{g^{-1}} - T U_{g g^{-1}} \in \{S \in \call(\cale)|\, a S, S a \in \calk(\cale) \mbox{ for all } a \in A\}$ for all
$g \in G$. The 
Kasparov group $KK^G(A,B)$ is defined to be the collection $\E^G(A,B)$ of these cycles
divided by homotopy induced by $\E^G(A,B[0,1])$.
}
\end{definition}

There exists an associative Kasparov product in $KK^G$ as usual (see \cite{burgiSemimultiKK}).

We list here some notions from \cite{thomsen}.
One is given two ungraded separable $G$-algebras $(A,\alpha)$ and $(B,\beta)$.
The $G$-algebra $B$ is assumed to be weakly stable, that is,
there exists a $G$-equivariant isomorphism $(B ,\beta) \cong (B \otimes \calk, \beta \otimes id_\calk)$.
The {\em multiplier} and {\em corona} algebra of $A$ is denoted by $\calm(A)$ and $Q(A):= \calm(A)/A$, respectively.

We remark that
the identity $\calm(A) \cong \call_A(A)$ for a $C^*$-algebra $A$ is often and freely used.
(We recall that the isomorphism is given by mapping an operator $T \in \call_A(A)$ to the double centralizer $(T,T') \in \calm(A)$,
where $T'(a):= (T^*(a^*))^*$ for all $a \in A$.)
In particular, 
$\calm(B \otimes \calk)$ and $\call_{B \otimes \calk}(B \otimes \calk)$ will be often identified.
\begin{definition}
{\rm
By using the $*$-isomorphism $\calm(A) \cong \call(A)$, the multiplier algebra
turns to a $G$-algebra $(\calm(A), \overline\alpha) \cong (\call(A),\overline \alpha)$ 
under the $G$-action
$\overline \alpha:G \rightarrow \mbox{End}(\call(A))$ determined
by $\overline \alpha_g(T) := \alpha_g \circ T \circ \alpha_{g^{-1}}$ for all $g \in G$ and $T \in \call(A)$
as in Kasparov \cite[1.4]{kasparov1980}.
}
\end{definition}

Thomsen interprets $\overline \alpha_g$ as the strictly continuous extension of 
$\alpha_g$ from $A$ to $\calm(A)$.
We shall however always use the aforementioned explicit formula
in our computations.
This $G$-action on the multiplier algebra induces a canonical $G$-action $\widehat \alpha:G \rightarrow \mbox{End}
(Q(A))$ on the corona algebra, which turns it to a $G$-algebra $(Q(A),\widehat \alpha)$.
The $G$-equivariant quotient map between multiplier and corona algebra is denoted by
$q_A:\calm(A) \rightarrow Q(A)$.
A $G$-equivariant $*$-homomorphism $\varphi:A \rightarrow B$ which is quasi-unital
(i.e. $\overline{\mbox{span}}\,{\varphi(A) B} = p B$ for some projection $p \in \calm(B)$)
induces a well-known 
strictly continuous $*$-homomorphism $\overline \varphi:\calm(A) \rightarrow \calm(B)$
between the multiplier algebras, which once again induces a canonical $G$-equivariant $*$-homomorphism
$\widehat \varphi:Q(A)\rightarrow Q(B)$.

\begin{definition}   \label{defGinvariantoperator}
{\rm
An operator $T \in \call(\cale)$ on a $G$-Hilbert module $\cale$ with $G$-action $U$ is called {\em $G$-invariant}
if
$T \circ U_g = U_g \circ T$ (equivalently, $U_g T U_{g^{-1}} = T U_{g} U_{g^{-1}}$
or $U_g T U_{g^{-1}} = U_{g g^{-1}} T U_{g g^{-1}}$) for all $g \in G$.
}
\end{definition}

In case of a multiplier algebra, $G$-equivariance is also called as follows. 

\begin{definition}  \label{defbetainvariance}
{\rm
An operator in $T \in \calm ((B,\beta))$ is called {\em $\overline \beta$-invariant} if $T$
(and consequently $T^*$) commutes with $\beta_g$ for all $g \in G$ (equivalently: $\overline \beta_g(T)
= T \overline \beta_g(1)$ or  $\overline \beta_g(T)
= \overline \beta_{g g^{-1}}(T)$ for all $g \in G$).
}
\end{definition}

The operators $V_1, V_2 \in \calm((B,\beta))$ always denote $\overline \beta$-invariant isometries
such that
$V_1 V_1^* + V_2 V_2^* = 1$ (cf. \cite[Lemma 4.1]{burgiUniversalKK}).
They are used to form $K$-theoretical addition in $\calm(B) \cong M_2(\calm(B))$.
Note that we also write sometimes simply $g$ 
for the action maps $\alpha_g, \beta_g, U_g$ etc.

In \cite[Section 2]{thomsen} $G$-equivariant extensions of $G$-algebras $A$ and $B$ are introduced and
identified with the set of $G$-equivariant $*$-homomorphisms 
from $A$ to $Q(B)$
$\mbox{Hom}_G(A,Q(B))$.
The theory of this section
goes 
essentially literally through.
The \cite[Theorem 2.1]{thomsen} can be ignored since it is only a continuity statement
and hence trivial under our setting.

\begin{theorem}
Discussion of \cite[Theorem 2.2]{thomsen}.
\end{theorem}

\begin{proof}
To prove the $G$-equivariance of a certain $*$-isomorphism $\kappa:E_1 \rightarrow E_2$
appearing in the diagram of \cite[Theorem 2.2]{thomsen}, we check first that it is $E$-equivariant.
Let $p \in E$ and write $p^{\bot}=1-p$. Then by the diagram of \cite[Theorem 2.2]{thomsen}, $p^\bot \kappa p(E_1) \subseteq j_2(B)$
and $p \kappa^{-1} p^\bot(j_2(B))=\{0\}$. Hence, $p \kappa^{-1} p^\bot \circ p^\bot \kappa p = 0$.
Notice that $\kappa^{-1} p^\bot \kappa$ is a self-adjoint projection in $\call_{E_1}(E_1)$, and hence commutes
with $p$, whence $p \kappa^{-1} p^\bot = 0$. Similarly $p^\bot \kappa^{-1} p = 0$ and so $\kappa^{-1}$
intertwines $p$.

Let $g \in G$.
We set then, as in Thomsen's proof, $\mu:=\kappa^{-1} \circ \gamma^2_g \circ \kappa \circ \gamma^1_{g^{-1}} : g g^{-1} E_1 \rightarrow g g^{-1} E_1$
(also the maps $j_1$ and $p_1$ in the diagram of the proof of \cite[Theorem 2.2]{thomsen} have to be restricted to $j_1:g g^{-1} B \rightarrow
g g^{-1} E_1$ and $p_1: g g^{-1} E_1 \rightarrow g g^{-1} A$)
and deduce that $\mu = id_{g g^{-1} E_1}$.
Hence $\gamma^2_g \circ \kappa = \kappa \circ \gamma^1_g$.
\end{proof}

%

\section{The equivariant extension groups}

Summing up \cite[Section 3]{thomsen}, for given $G$-algebras $A$ and $B$ an addition is introduced on the
set $\mbox{Hom}_G(A,Q(B))$ (which, recall, encodes extensions) and equivalence relations are imposed on it.
Two abelian extensions groups $\mbox{Ext}_G(A,B)$ and $\mbox{Ext}_G^h(A,B)$ come out, and it is shown
that the canonical map $\mbox{Ext}_G(A,B) \rightarrow \mbox{Ext}_G^h(A,B)$ is surjective.

In this section we have to modify the definition of unitary equivalence as follows.

\begin{definition}    \label{defUnitarilyEqu}
{\rm
Two $G$-extensions $\varphi,\psi \in \mbox{Hom}_G(A,Q(B))$ are said to be {\em unitarily equivalent}
if there exists a unitary $u \in \calm(B)$ such that for all $g \in G$

1. $\widehat{\beta}_g \big(q_B(u) \big) \,\,=\,\, \widehat \beta_{g g^{-1}} \big( q_b(u) \big)$,

2. $\big (\mbox{Ad } q_B(u) \big) \circ \varphi \,\, =  \,\, \psi$.

}
\end{definition}

Note that we made the only obvious modification in Definition \ref{defUnitarilyEqu} that
$\widehat  \beta_{g g^{-1}}$ was added in point 1.
The construction of certain equivariant paths of isometries and unitaries
in \cite[Lemma 3.3]{thomsen}
work also inverse semigroup equivariantly. (Note that the unitary here has nothing to do
with the fact that group actions are realized by unitaries.)
In the statements and proofs of Lemma 3.2, Lemma 3.4 and Proposition 3.5 of Thomsen's paper the identities
$\widehat \beta_g(q_B(X)) = q_B(X)$
and $\overline \beta_g(X) = X$, respectively, for $X=V,U,S,W,W_2,U_t,u$
have to be replaced by $\widehat \beta_g(q_B(X)) = \widehat \beta_{g g^{-1}} (q_B(X))$ and
$\overline \beta_g(X) = \overline \beta_{g g^{-1}}( X)$, respectively.
All these modifications are obvious.
The extra terms $\widehat \beta_{g g^{-1}}$
and $\overline \beta_{g g^{-1}}$ may then slightly alter some computations, but since we shall demonstrate
various similar computations
in the next sections we omit presenting further details in this section.
%


\section{An appropriate picture of $KK^1$}

Summarizing Section 4 of \cite{thomsen}, the notion of the one-sided graded Kasparov group
$KK_G^1(A,B):= KK^G(A,B \oplus B)$ is defined, and another group $\widetilde{KK}_G^1(A,B)$
of equivariant $A,B$-$KK^1$-cocycles is introduced. An isomorphism $\widetilde{KK}_G^1(A,B) \cong KK_G^1(A,B)$
between both groups is then established.

In this section
Thomsen recalls the definition of $KK^G$-theory, for which we use
Definition \ref{defCycle}.
The definition of a Kasparov cycle $(\cale,T) \in \E^G(A,B)$ requires that $(g(T) - T g g^{-1}(1)) a \in \calk(\cale)$
for all $a \in A$ and $g\in G$. Thomsen remarks that we can drop here $a$ at all and require the stronger version
\begin{equation}    \label{equkasparovstronger}
g(T) - T g g^{-1}(1) 
\;\in\; \calk(\cale)
\end{equation}
for all Kasparov cycles without changing $KK^G$-theory. This is true also in our
case as pointed out in \cite{burgiUniversalKK}.
In Section 4 Thomsen introduces the notion of a unitary cocycle, for which we have to use the following
modified version (as already pointed out in \cite{burgiUniversalKK}).

\begin{definition}	\label{defunitarycocycle}
{\rm
Let $(B,\beta)$ be a $G$-algebra.
A {\em unitary $\beta$-cocycle} is a map $u: G \rightarrow \calm(B)$
such that
$$\beta_{g g^{-1}} = u_g^* u_g, \quad u_{g g^{-1}} = u_g u_g^*, \quad u_{g h} = u_g \overline\beta_g(u_h)$$
for all $g, h \in G$.
}
\end{definition}
This definition implies that $u_g$ is a partial isometry with identical range and source projection.
Furthermore, we have the identities
\begin{equation}  \label{remarkcocycle}
\beta_{g g^{-1}} = u_g^* u_g = u_g u_g^* = u_{g g^{-1}} \quad \mbox{and}\quad \overline\beta_g(u_{g^{-1}}) = u_g^*
\end{equation}
for all $g \in G$, see \cite[Lemma 3.3]{burgiUniversalKK}.
Thus, in our setting it would be more natural to simply speak about ``$\beta$-cocycles" instead of
``unitary $\beta$-cocycles". 
Informally we remark that a unitary $\beta$-cocycle will be used to combine it with the $G$-action $\overline \beta$
to form a new $G$-action $g \mapsto (\mbox{Ad } u_g) \circ \overline \beta_g$ on $\calm(B)$.
Cocycles come into play as ballast under some transformations of Kasparov cycles
which are not $G$-equivariant, 
notably in \cite[Theorem 4.3]{thomsen} as
$u: G \rightarrow \calm(B)$.

The definition of an equivariant $A,B$-$KK^1$-cocycle changes slightly:
%
%


\begin{definition}   \label{defKKcocycle}
{\rm
An {\em equivariant $A,B$-$KK^1$-cocycle} is a triple $(\pi,v,p)$ consisting of a
$*$-homomorphism $\pi:A \rightarrow \calm(B)$, a unitary $\beta$-cocycle
$v: G \rightarrow \calm(B)$ (Definition \ref{defunitarycocycle}) and a projection $p \in \calm(B)$ such that
\begin{itemize}
\item[1.] $\big ( \mbox{Ad } v_g \circ \overline\beta_g \big) \big( \pi  (a) \big ) \; = \; \pi \big (\alpha_g(a) \big )
\qquad (\forall a \in A, g \in G)$

\item[2.]
$\big (\mbox{Ad } v_g \circ \overline\beta_g \big) (p) \,\, - \,\, \overline \beta_{g g^{-1}} (p)  \,\, \in \,\,  B \otimes \calk  \qquad (\forall g \in G)$

\item[3.]
$p \pi(a) - \pi(a) p \,\, \in \,\, B \otimes \calk \qquad (\forall a \in A)$.

\end{itemize}
}
\end{definition}


Only the additional operator $\overline \beta_{g g^{-1}}$ in the second point and the usage
of the modified notion of $\beta$-cocycles have changed.
These $KK^1$-cocycles are also called triples in \cite{thomsen}.
The definition of a degenerate cocycle has to be similarly slightly and obviously adapted as follows.

\begin{definition}   \label{defDegKKcocycle}
{\rm
An equivariant $A,B$-$KK^1$-cocycle $(\pi,v,p)$
is called {\em degenerate} when
\begin{eqnarray}
\label{s1}
\big( \mbox{Ad } v_g \circ \overline \beta_g \big ) (p) & = & 
\overline \beta_{g g^{-1}}(p)
\qquad(\forall g \in G),\\
\nonumber   p \pi(\cdot) &=& \pi(\cdot) p.
\end{eqnarray}
}
\end{definition}

The sum of $KK^1$-cocycles is realized by any pair $(V_1,V_2)$ of $G$-invariant isometries
in $\calm(B)$ such that $V_1 V_1^* + V_2 V_2^* = 1$. 


\begin{lemma}
In \cite[Lemma 4.2]{thomsen} Thomsen shows that the sum of an $A,B$-$KK^1$-cocycle with a degenerate $A,B$-$KK^1$-cocycle is
homotopic to itself.
\end{lemma}

%
%
%

\begin{proof}
In accordance to (\ref{s1}), 
we only have to alter the line
$$\lim_{t \rightarrow 0} \,\mbox{Ad } [S_t u_g S_t^* + T_t v_g T_t^*] \circ \overline\beta_g (S_t q S_t^* + T_t p T_t^*)
- (S_t q S_t^* + T_t p T_t^*)$$
in the proof of \cite[Lemma 4.2]{thomsen} to
$$\lim_{t \rightarrow 0} \,\mbox{Ad } [S_t u_g S_t^* + T_t v_g T_t^*] \circ \overline\beta_g (S_t q S_t^* + T_t p T_t^*)
- \overline\beta_{g g^{-1}} (S_t q S_t^* + T_t p T_t^*).$$
\end{proof}

\begin{lemma}
After \cite[Lemma 4.2]{thomsen}, Thomsen defines a map $\Phi: \widetilde{KK}^1_G(A,B) \rightarrow KK^1_G(A,B)$
by setting $\Phi[\pi,v,p] = [E_B,\tilde \pi,S_p]$, where $E_B:= B \oplus B$, $\tilde \pi = \pi \oplus \pi$ and
$S_p(b_1,b_2) = ((2p-1) b_1,(1-2p)b_2)$ for $b_1,b_2 \in B$.
The Hilbert module $E_B$ is endowed with the $G$-action $\tilde v_g := v_g \beta_g$.
\end{lemma}

\begin{proof}
Let us check that $[E_B,\tilde \pi,S_p]$ is a Kasparov cycle, where we only need to bother about the $G$-action.
That $(E_B,\tilde \pi)$ is a $G$-Hilbert bimodule we have already discussed in \cite{burgiUniversalKK}.
So let us only check identity $g(S_p) - S_p g(1) \in \calk(E_B)$ from (\ref{equkasparovstronger}).
We have, by considering only the first factor $B$ of $E_B$,
\begin{eqnarray*}
\tilde v_g S_p \tilde v_{g^{-1}}(b_1) &=& v_g \beta_g\big ( (2p-1) v_{g^{-1}} \beta_{g^{-1}}(b_1) \big )  \\
&=& 2 v_g \overline\beta_g (p) \overline\beta_g(v_{g^{-1}}) \beta_{g g^{-1}}(b_1)
- v_g \overline\beta_g(v_{g^{-1}}) \beta_{g g^{-1}}(b_1)   \\
&=& 2 \big (\mbox{Ad }v_g  \circ \overline\beta_g \big ) (p) \, \beta_{g g^{-1}}(b_1)
- v_g v_g^* \beta_{g g^{-1}}(b_1) \\
&\equiv& 2p \beta_{g g^{-1}}(b_1)   - \beta_{g g^{-1}}(b_1) 
\quad = \quad S_p \tilde v_g \tilde v_{g^{-1}}(b_1)
\end{eqnarray*}
for all $b_1 \in B$,
because
$\overline \beta_g(v_{g^{-1}}) = v_g^*$, $v_g v_g^* = \beta_{g g^{-1}}$ (see (\ref{remarkcocycle})), by Definition \ref{defKKcocycle}
and by using that $\overline \beta$ is the extension of $\beta$ from $B$ to $\calm(B)$.
\end{proof}

The proof of \cite[Theorem 4.3]{thomsen} is very similar to that of \cite[Theorem 3.5]{thomsen2}
and was discussed inverse semigroup equivariantly in \cite[Theorem 4.4]{burgiUniversalKK}, so we omit a rediscussion.

\begin{lemma}
We discuss \cite[Lemma 4.4]{thomsen}.
\end{lemma}

\begin{proof}
Thomsen notes that an application of Kasparov's \cite[Lemma 6.1]{kasparov1981} yields a unitary
in $T \in M_2(\calm(B))$ which is $G$-invariant in Kasparov's case but only $G$-invariant modulo compacts
in Thomsen's case.
Similarly, 
when inspecting Kasparov's proof of his lemma (the $G$-action on $\calm(B)$ is $v \beta$
and $T$ is derived from the 
operator of a Kasparov cycle), in our case $G$-invariance of $T$ means
(according to Definition \ref{defCycle}) that
$$
\quad
\left (
\begin{matrix}
	v_g & 0 \\
	0  & w_g^2
\end{matrix}
\right )
\overline{\beta_g \otimes id_{M_2}}(T)
\left (
\begin{matrix}
	v_{g}^* & 0 \\
	0  & {(w_{g}^2)}^*
\end{matrix}
\right )
\; - \; T
\left (
\begin{matrix}
	v_{g} \beta_{g g^{-1}} v_{g}^* & 0 \\
	0  & w_{g}^2 \beta_{g g^{-1}} {(w_{g}^2)}^*
\end{matrix}
\right )
$$
lies in $M_2(B)$.
Recall that $v_g$ and $w_g^2$ are partial isometries with source projections 
$v_{g}^* v_g = \beta_{g g^{-1}} = \overline\beta_g(1) = (w_{g}^2)^* w_g^2 \in \calm(B)$ in the center of $\calm(B)$ by (\ref{remarkcocycle}).
%
%
Hence we then also obtain easily relation ``3." of \cite[Lemma 4.4]{thomsen}.
%
\end{proof}


\section{Twisted $G$-extensions and $KK^1$}

The aim of \cite[Section 5]{thomsen} is the introduction of a twisted extension group $\mbox{Ext}_{G,t}(A,B)$
and the establishment of an isomorphism $\mbox{Ext}_{G,t}(A,B) \cong \widetilde{KK}_G^1(A,B)$.
Moreover, is is pointed out that two elements in $\mbox{Ext}_{G,t}(A,B)$ are identical if and only if their representatives are homotopic.

The notions of twisted extensions, unitary equivalence of twisted extensions,
and degenerate twisted extensions
need not be formally altered, excepting the
implicitly self-evident fact
that we have to use the modified notion of $\beta$-cycles of Definition \ref{defunitarycocycle}.
In the proofs of this section we have however slight adaption.

At first let us observe that unitary equivalence between two twisted extensions
$(\varphi,u)$ and $(\psi,v)$, where $\varphi,\psi:A \rightarrow Q(B)$ are $G$-equivariant $*$-homomorphisms
and $u, v$ are $\beta$-cocycles, is an equivalence relation.
Equivalence is defined through the existence of
a unitary $u \in \calm(B)$ such that $\mbox{Ad}\, q_B(u) \circ \varphi= \psi$
and $v_g \overline \beta_g(u) - u u_g \in B$ for all $g \in G$.
Let us demonstrate symmetry of the second relation. Taking the second relation for granted, we get
\begin{eqnarray*}
B &\ni& \overline \beta_{g^{-1}} (v_g \overline \beta_g(u) - u u_g)^* \;=\;
\overline  \beta_{g^{-1}} ( \overline \beta_g(u^*) \overline \beta_g(v_{g^{-1}})
- \overline \beta_g(u_{g^{-1}}) u^* ) \\
&=& u^* \overline \beta_{g^{-1}g}( v_{g^{-1}})
- u_{g^{-1}} \overline \beta_{g^{-1}}( u^* )  \;=\;
u^* v_{g^{-1}}
- u_{g^{-1}} \overline \beta_{g^{-1}}( u^* )
\end{eqnarray*}
by identities (\ref{remarkcocycle}), the compatibility of the $\overline \beta$-action, and the fact
that $\overline \beta_{g^{-1}g}( v_{g^{-1}}) = \beta_{g^{-1}g} v_{g^{-1}} \beta_{g^{-1}g}  =
v_{g^{-1}}$ by $\beta_{g^{-1}g}= v_{g^{-1}} v_{g^{-1}}^*$ by identities (\ref{remarkcocycle}), as required.

\begin{proposition}
Between \cite[Lemma 5.1]{thomsen} and \cite[Lemma 5.2]{thomsen}, Thomsen shows that there is a group homomorphism
$\Lambda: \mbox{Ext}_{G,t}\big((A,\alpha),(B,\beta)\big) \longrightarrow \widetilde{KK}_G^1(A,B)$.
\end{proposition}

\begin{proof}
%
%
%
%
%
Thomsen considers a twisted extension $(\varphi,u) \in \mbox{Ext}_{G,t}(A,B)$, where
$\varphi:A \rightarrow Q(B)$ is a $G$-equivariant $*$-homomorphism and $u$ a $\beta$-cocycle.
Then an inverse twisted extension $(\psi,v)\in \mbox{Ext}_{G,t}(A,B)$ to this one is chosen, that is, there
is a degenerate twisted extension $(\lambda,w)\in \mbox{Ext}_{G,t}(A,B)$ such that
$(\varphi+\psi,u+v)$ is unitarily equivalent to $(\lambda,w)$.
One deduces that there exists a unitary $T \in \calm(B)$ such that
\begin{eqnarray*}
&& \mbox{Ad } q_b(T) \circ (\varphi +\psi) = \lambda,\\
&& w_g \overline\beta_g(T) - T(V_1 u_g V_1^* + V_2 v_g V_2^*) \quad \in B \quad (\forall g \in G)
\end{eqnarray*}
in Thomsen's paper as well as here as
the formal definitions around twisted extensions are unchanged.
%
%
Then Thomsen states that one has
\begin{eqnarray}
&& \mbox{Ad} \, w_g \circ \overline\beta_g(T V_1 V_1^* T^*)   \nonumber \\
&& = \,\,\, \mbox{Ad} \, T \circ \mbox{Ad} (u+v)_g \circ \overline\beta_g(V_1 V_1^*) \,\, = \,\, T V_1 V_1^* T^* \quad \quad (g \in G) \quad    \label{mod1}
\end{eqnarray}
in $\calm(B)$ modulo $B$.
This changes in our setting. 
%
We compute in $\calm(B)$ modulo $B$
\begin{eqnarray*}
&& \mbox{Ad} \, w_g \circ \overline\beta_g(T V_1 V_1^* T^*) \,\, = \,\,  w_g \overline\beta_g(T) \,
 V_1 V_1^* \, {\overline\beta_g(T)}^*  w_g^*    \\
&=& T (V_1 u_g V_1^* + V_2 v_g V_2^*) \,\, V_1 V_1^*
\,\,   (V_1 u_g^* V_1^* + V_2 v_g^* V_2^*) T^*    \\
&=& T V_1 u_g u_g^* V_1^*  T^*
\,\,= \,\, \beta_{g g^{-1}} T V_1 V_1^*  T^*  \beta_{g g^{-1}}
\,\, = \,\, \overline\beta_{g g^{-1}} (T V_1 V_1^*  T^*),
\end{eqnarray*}
because $\beta_{g g^{-1}} = u_g  u_g^*$, see (\ref{remarkcocycle}), is in the center of $\calm(B)$.
So we have an additional $\overline\beta_{g g^{-1}}$ in identity (\ref{mod1}).
Hence $(\overline \lambda,w,T V_1 V_1^* T^*)$ is a $KK^1$-cocycle
in the sense of Definition \ref{defKKcocycle}, where $\overline \lambda:A \rightarrow \calm(B)$
is so chosen that $\lambda = q_B \circ \overline \lambda$ and $\mbox{Ad}\, w_g \circ \overline \beta_g \circ
\overline \lambda = \overline \lambda \circ \alpha_g$. 
This $KK^1$-cocycle can be interpreted as an element
in $\widetilde{KK}^1_G(A,B)$ (homotopy classes of $KK^1(A,B)$-cocycles).

To see the well-definedness of the just constructed assignment $\mbox{Ext}_{G,t}(A,B) \rightarrow \widetilde{KK}^1_G(A,B)$,
Thomsen considers another pair $(\psi^1,v^1),(\lambda^1,w^1) \in \mbox{Ext}_{G,t}(A,B)$
instead of $(\psi,v),(\lambda,w) \in \mbox{Ext}_{G,t}(A,B)$.
To the other pair is associated a unitary $T^1 \in \calm(B)$ instead of $T$.
There are shown the existence of homotopies
$$(\overline{\lambda^1}, w^1, T^1 V_1 V_1^* {T^1}^*) \sim (\mu,w^2, T V_1 V_1^* T^*)
\sim (\overline{\lambda}, w, T V_1 V_1^* {T}^*),$$
where $S= T {T^1}^*$, $\mu= \mbox{Ad}\, S \circ \overline{\lambda^1}$ and $w^2=S w_g^1 \overline\beta_g(S^*)
\in \calm(B)$.
Let us check that $(\mu,w^2, T V_1 V_1^* T^*) \in \widetilde{KK}^1_G(A,B)$, so is a triple
in the sense of Definition \ref{defKKcocycle}.
Using that $\beta_{g g^{-1}}$ is in the center of $\calm(B)$,
we have
\begin{eqnarray*}
&& \mbox{Ad}\, w_g^2 \circ \overline\beta_g \circ \mu(a) \;=\;
\mbox{Ad}\, S w_g^1 \overline\beta_g(S^*) \circ \overline\beta_g \circ \mbox{Ad}\, S \circ \overline{\lambda^1}(a) \\
&=&
S w_g^1 \beta_g S^* \beta_{g^{-1}} \beta_g S \overline{\lambda^1}(a) S^*\beta_{g^{-1}} \beta_g S \beta_{g^{-1}}
{w_g^1}^* S^*
\;=\; \mu(\alpha_g(a)), \mbox{ and}\\
&&\mbox{Ad} \, w_g^2 \circ \overline \beta_g(T V_1 V_1^* T^*)
\; = \; S w_g^1 \beta_g S^* \beta_{g^{-1}} \beta_g T V_1 V_1^* T^* \beta_{g^{-1}}
\beta_g S \beta_{g^{-1}} {w_g^1}^* S^*   \\
&=&  T {T^1}^* w_g^1 \beta_g {T^1} T^* T V_1 V_1^* T^* T {T^1}^* \beta_{g^{-1}} {w_g^1}^* {T^1} T^*\\
&\equiv&  T {T^1}^* \overline\beta_{g g^{-1}}({T^1} V_1 V_1^* {T^1}^*) {T^1} T^*
\;=\; \overline\beta_{g g^{-1}}({T} V_1 V_1^* {T}^*) \quad \mod B \otimes\calk.
\end{eqnarray*}
This proves the first two relations of Definition \ref{defKKcocycle}, and the third one is similar.
The verification of the remaining parts is technically very similar to the last computations and thus we omit 
a further discussion.
\end{proof}

\cite[Lemma 5.2]{thomsen} goes through unchanged. 

\begin{theorem}
In \cite[Theorem 5.3]{thomsen} it is remarked that the map $\Lambda:\mbox{Ext}_{G,t}((A,\alpha),(B,\beta)) \rightarrow \widetilde{KK}^1_G(A,B)$
is an isomorphism.
\end{theorem}

\begin{proof}
The proof goes through unchanged.
We only inspect the surjectivity proof, where from a given $A,B$-$KK^1$-cocycle $(\pi,v,p)$
a $*$-homomorphism $\varphi:A \rightarrow Q(B)$ via $\varphi(a)=q_B(\pi(a)p)$
and then a twisted extension $(\varphi,v) \in \mbox{Ext}_{G,t}(A,B)$ is defined.
Checking that it is a twisted extension, we compute
\begin{eqnarray*}
&& \mbox{Ad}\, q_B(v_g) \circ \widehat \beta_g \circ \varphi(a) \; = \;
q_B(\mbox{Ad}\,v_g  \,\overline \beta_g (\pi(a) p)) \; = \;
q_B(v_g  \,\overline \beta_g (\pi(a)) v_g^* \, \overline\beta_{g g^{-1}} (p))\\
&=& q_B(v_g  \,\overline \beta_g (\pi(a)) v_g^* )\, q_B(p)
\;=\; \varphi(\alpha_g(a))
\end{eqnarray*}
by Definition \ref{defKKcocycle} and $\beta_{g g^{-1}} = v_g v_g^*$ (identities (\ref{remarkcocycle})).
\end{proof}

Also \cite[Theorem 5.5]{thomsen} goes essentially through unchanged and involves only similar verifications
and computations already having been demonstrated. 


\section{Removing the twist}

The aim of this section is the construction 
of a connection between twisted and untwisted extension
groups, see Lemma \ref{lemmaafterlemma62} below.
Here the paper \cite{mingophillips} of Mingo and Phillips is essential 
and its adaption to inverse semigroups is presented in Section \ref{sectionMingoPhillips}
as an appendix.


\begin{lemma}[Cf. Lemma 6.1 of \cite{thomsen}]     \label{lemma61}
We discuss here \cite[Lemma 6.1]{thomsen}, whose statement changes slightly.
%
Obviously, in the statement of \cite[Lemma 6.1]{thomsen} we have to use the
identities
\begin{eqnarray*}
\widehat{\beta_g} \big(q_B(u) \big)\psi_1(a)
&=& \widehat{\beta_{g g^{-1}}} \big(q_b(u) \big) \psi_1(a), \\
\widehat{\beta_g \otimes id_{M_4}}\big(q_{M_4(B)}(U_t) \big) &=& \widehat{\beta_{g g^{-1} \otimes id_{M_4}}} \big (q_{M_4(B)}(U_t) \big)
\end{eqnarray*}
instead of the corresponding identities in Thomsen's paper without the appearances of
the $\beta_{g g^{-1}}$ expressions.
\end{lemma}

\begin{proof}
In the proof of \cite[Lemma 6.1]{thomsen} the identity
\cite[(6.1)]{thomsen} changes to
$$\widehat{\beta_g \otimes id_{M_2}}\big(q_D(S) \big)\psi_1(a)
= \widehat{\beta_{g g^{-1}} \otimes id_{M_2} } \big(q_D(S) \big) \psi_1(a).$$

The function $\varphi$ appearing in \cite{thomsen} takes the new form $\varphi(g) = \overline{\beta_g \otimes id_{M_2}}(S)
- \overline{\beta_{g g^{-1}} \otimes id_{M_2}}(S)$.
Thomsen then applies Kasparov's technical \cite[Theorem 1.4]{kasparov1988},
for which we use the corresponding and quite similar technical \cite[Theorem 5.1]{burgiSemimultiKK} (the function $\psi$ there has to be ignored).

In item ``2." of the proof of \cite[Lemma 6.1]{thomsen}
we must replace the occurrences $\overline{\beta_g \otimes id_{M_2}}(Y) X
- Y X$
by
$$\overline{\beta_g \otimes id_{M_2}}(Y) X
- \overline{\beta_{g g^{-1}} \otimes id_{M_2}}(Y) X$$
for $Y=S,SS$ and $X= x,xS$.
The application of Kasparov's technical theorem yields positive elements $m,n \in \calm(D)$
satisfying
\begin{eqnarray*}
&& \overline{\beta_g \otimes id_{M_2}}(m) - \overline{\beta_{g g^{-1}} \otimes id_{M_2}} (m), \;
\overline{\beta_g \otimes id_{M_2}}(S) n - \overline{\beta_{g g^{-1}} \otimes id_{M_2}} (S) n,\\
&& n \overline{\beta_g \otimes id_{M_2}}(S) - n \overline{\beta_{g g^{-1}} \otimes id_{M_2}}(S) \;\in \; D
\end{eqnarray*}
rather than the corresponding relations without the appearance
of $\overline{\beta_{g g^{-1}} \otimes id_{M_2}}$ as in Thomsen's paper.
\end{proof}

%

{\em Assume from now on that $G$ is $E$-continuous!}

For the definition of $E$-continuity see Definition \ref{definitionEcontinuous}
or Lemma \ref{lemmaEcontinuity}.
From now on, everywhere where Thomsen uses the $G$-Hilbert space $L^2(G)$ with its right regular representation
we have to use the $G$-Hilbert $C_0(X)$-module $\widehat{\ell^2}(G)$ defined in Definition \ref{defCompatibleL2}
(there we need $E$-continuity). 
For simplicity, we shall thus also switch to the notation $L^2(G):=\widehat{\ell^2}(G)$. 
The actions on $\call(L^2(G))$ and $\calk(L^2(G))$ are the usual induced ones.
For every $G$-Hilbert $B$-module or $G$-algebra $\cale$, $L^2(G) \otimes \cale$ in Thomsen \cite{thomsen}
has to be replaced by the compatible tensor product
$L^2(G) \otimes^X \cale$ (cf. Definition \ref{defL2hilbertmodule}), and likewise $\calk(L^2(G)) \otimes \cale$ in Thomsen \cite{thomsen} by
$\calk(L^2(G)) \otimes^X \cale$ here.
Every occurrence of the Hilbert space $\C$ in \cite{thomsen} has to be substituted by the $G$-Hilbert $C_0(X)$-module $C_0(X)$. This also includes, that we shall substitute the compact operators on a separable Hilbert space, $\calk$,
with its trivial $G$-action
by the $G$-algebra $\calk(C_0(X)^\infty)$ and write the $G$-algebra $\calk \otimes A$
isomorphically as
$\calk \otimes A \cong \calk(C_0(X)^\infty) \otimes^X A$ as $G$-algebras.
(The isomorphism is of course given by $e_{ij} \otimes a \mapsto \theta_{1_{C_0(X)}^{(i)},1_{C_0(X)}^{(j)}} \otimes a$
(position $i$ and $j$).)

Analogously as in Thomsen \cite{thomsen} one defines the $G$-algebra
$$\calk_G  \,:= \, \calk \big(C_0(X)^\infty \big) \otimes^X \calk \big(L^2(G) \big).$$

For convenience of the reader, and since we are using a different $L^2(G)$-space
in our setting which requires higher attention if indeed everything works, we 
shall lay out more details
and larger parts from Thomsen's proof than above 
in the rest of this section.

\begin{lemma}[Cf. Lemma 6.2 of \cite{thomsen}]    \label{lemma62}
Let $[(\varphi,u)] \in \mbox{Ext}_{G,t}(A,B)$ be a twisted extension.
Write $\rho$ for the $G$-action on $\calk_G$. 
Then there is a unitary $w \in  \calm(B {\otimes}^X \calk_G)$
such that
\begin{equation}    \label{comp3}
w (u_g \otimes 1_{\calk_G})(\beta_g \otimes \rho_g) = (\beta_g \otimes \rho_g) w .
\end{equation}
\end{lemma}

\begin{proof}
Let the $G$-action on $L^2(G)$ (see Definition \ref{defCompatibleL2}) 
be denoted by $\tau$.
We have two $G$-actions on $B^\infty \otimes^X L^2(G)$, namely $\mu_g :=(u_g \beta_g)^\infty \otimes \tau_g$
and $\nu_g:= (\beta_g)^\infty {\otimes} \tau_g$ for $g \in G$. The two corresponding $G$-Hilbert $B$-modules
$B^\infty \otimes^X L^2(G)$
are $G$-equivariantly isomorphic by Lemma \ref{lemmamingophillips23}.
Hence, there exists a unitary $v \in \call_B(B^\infty \otimes^X L^2(G))$ intertwining the aforementioned two $G$-actions.
Consider the isomorphism of Hilbert $B$-modules, see Kasparov \cite[Theorem 2.1]{kasparov1980},
\begin{eqnarray*}     
\gamma:\call_B(L^2(G) \otimes^X B^\infty) \rightarrow \calm \big(\calk \big(L^2(G)
\otimes^X B^\infty \big) \big):
\gamma(T)=(T_1,T_2),
\end{eqnarray*}
where $(T_1,T_2)$ is the double centralizer given by
$T_1(\theta_{x,y})=\theta_{Tx,y}$ and $T_2(\theta_{x,y})=\theta_{x,T^* y}$
for $x,y \in L^2(G) \otimes^X B^\infty$. 
%

Note that $\gamma$ is $G$-equivariant for both actions induced by $\mu$ and $\nu$, respectively, on the domain of $\gamma$,
and their corresponding
induced
actions on the range of $\gamma$, respectively, since
$${(g(T))}_1 (\theta_{x,y}) = \theta_{g T g^{-1} x,y} = \theta_{g T g^{-1} x, g g^{-1} y}
= g \circ T_1 \circ g^{-1} (\theta_{x, y})$$
by compatibility of the inner product (i.e. $g g^{-1} \langle x,y \rangle =
\langle g g^{-1} x,y \rangle$) and the identity
$$g (\theta_{x,y})(z) = g (x \langle y,g^{-1}(z)\rangle)
= \theta_{gx,gy} (z)$$
for all $g \in G$ and $x,y,z \in L^2(G) \otimes^X B^\infty$.
Notice that
\begin{eqnarray*}    
&& \calk(L^2(G) \otimes^X B^\infty) \cong \calk(L^2(G)) \otimes^X \calk(B^\infty) \cong
\calk(L^2(G)) \otimes^X (\calk \otimes B) \\
&\cong& \calk(L^2(G)) \otimes^X \big (\calk(C_0(X)^\infty) \otimes^X B \big)
\cong \calk_G \otimes^X B
\end{eqnarray*}
isomorphically as $G$-algebras by 
P.-Y. Le Gall \cite[Proposition 4.2.(a)]{legall1999}.
We may thus identify the range of $\gamma$ with $\calm(B \otimes^X \calk_G)$.

Set $w:= \gamma(v)$.
Recall that $v \mu_g = \nu_g v$.
Hence $\nu_g v \nu_{g^{-1}} = v \mu_g \nu_{g^{-1}}$.
Since $\mu_g \nu_{g^{-1}} = (u_g \otimes 1) (\beta_{g g^{-1}} \otimes  \rho_{g g^{-1}})$,
we obtain from
$\gamma(\nu_g v \nu_{g^{-1}}) = \gamma(v) \gamma(\mu_g \nu_{g^{-1}})$ that
$$(\beta_g \otimes  \rho_{g}) w (\beta_{g^{-1}} \otimes  \rho_{g^{-1}}) = w (u_g \beta_{g g^{-1}} \otimes  \rho_{g g^{-1}}).$$
Multiplying from right with $\beta_g \otimes  \rho_{g}$, and noting that
$\beta_{g^{-1} g} \otimes  \rho_{g^{-1}g}$ is in the center, yields the claim.
%
%
%
%
%
\end{proof}

\begin{definition}[Cf. \cite{thomsen} after Lemma 6.2]    \label{afterlemma62}
{\rm
For a $G$-algebra $(C,\gamma)$ and a homomorphism $\varphi \in \mbox{Hom}_G(A, Q(B))$
we denote by $\varphi \hat \otimes id_C \in \mbox{Hom}_G(A \otimes^X C, Q(B \otimes^X C))$
the canonical map $\varphi \hat \otimes id_C (a \otimes c)= q_B^{-1}(\varphi(a)) \otimes c + B \otimes^X C$.
%
%
}
\end{definition}

\begin{lemma}[Cf. \cite{thomsen} after Lemma 6.2]   \label{lemmaafterlemma62}
For every twisted extension $[(\varphi,u)] \in \mbox{Ext}_{G,t}(A,B)$
choose a unitary $w \in \calm(B \otimes^X \calk_B)$ according to Lemma \ref{lemma62}. Then there exists a 
homomorphism
$$\Theta: \mbox{Ext}_{G,t}\big((A,\alpha),(B,\beta)\big) \rightarrow \mbox{Ext}_G \big (A
\otimes^X \calk_G, Q(B \otimes^X
\calk_G) \big)$$
defined by $\Theta([\varphi,u])= [\mbox{Ad}\, q_{B \otimes^X \calk_G}(w) \circ (\varphi
\hat \otimes id_{\calk_G})]$.

\end{lemma}

\begin{proof}
%
Recall from Section 5 of \cite{thomsen} that $\mbox{Ext}_{G,t}$ consists of twisted extension divided
by the equivalence relation $\sim$ that two twisted extensions $e_1,e_2$ are equivalent
if adding both with two (possibly different) degenerate twisted extensions, then they are unitarily equivalent,
and finally considering in this quotient space only invertible extensions with respect to addition (direct sum of extensions).
The extension group $\mbox{Ext}_G$ is defined analogously, without the word ``twisted" everywhere.
%

At first we are going to prove that $\Theta(\varphi,u)$ is indeed in $\mbox{Ext}_G$.
For $T \in \calm (B \otimes^X \calk_G)$ we have
\begin{eqnarray}
&& \mbox{Ad}(w) \circ \mbox{Ad}(u_g \otimes 1_{\calk_G}) \circ\overline{\beta_g \otimes \rho_g} \;(T) \nonumber \\
&=& w (u_g \otimes 1_{\calk_G})(\beta_g \otimes \rho_g) T
(\beta_{g^{-1}} \otimes \rho_{g^{-1}}) (u_{g}^* \otimes 1_{\calk_G}) w^* \nonumber  \\
&=& (\beta_g \otimes \rho_g) w T w^* w
 (u_{g^{-1}} \otimes 1_{\calk_G}) (\beta_{g^{-1}} \otimes \rho_{g^{-1}}) w^* \nonumber  \\
&=& (\beta_g \otimes \rho_g) w T w^*
 (\beta_{g^{-1}} \otimes \rho_{g^{-1}})  \nonumber  \\
&=&\overline{\beta_g \otimes \rho_g} \circ  \mbox{Ad}(w)\;(T)   \label{comp1}
\end{eqnarray}
in $\calm(B \otimes^X \calk_G)$
by Lemma \ref{lemma62} and (\ref{remarkcocycle}).
Since $(\varphi,u)$ is a twisted extension, and so $\mbox{Ad}\, q_B(u_g) \circ \widehat \beta_g \circ \varphi(a)
= \varphi(\alpha_g(a))$, see \cite[Section 5]{thomsen}, we get
\begin{eqnarray*}
\mbox{Ad}\, q_{B \otimes^X \calk_G}(u_g \otimes 1_{\calk_G})
\circ \widehat{\beta_g \otimes \rho_g} \circ (\varphi \hat \otimes id_{\calk_G})
&=& (\varphi \hat \otimes id_{\calk_G}) \circ (\alpha_g \otimes \rho_g).
\end{eqnarray*}
Applying $\mbox{Ad} q_{B \otimes^X \calk_G}(w)$ on this identity shows
with identity (\ref{comp1}) that
\begin{eqnarray}
\mbox{Ad} q_{B \otimes^X \calk_G}(w) \circ (\varphi \hat \otimes id_{\calk_G})
&\in& \mbox{Hom}_G(A \otimes^X \calk_G,Q(B \otimes^X \calk_G)),
\label{comp7}
\end{eqnarray}
in other words, this map is $G$-equivariant.
Hence, $\Theta(\varphi,u)$ is in $\mbox{Ext}_G$.

If $(\hat \varphi,u)$ is a degenerate twisted extension, that is,
$\mbox{Ad}(u_g) \circ \overline \beta_g \circ \varphi(a)
= \varphi(\alpha_g(a))$ with $\varphi \in \mbox{Hom}_G(A,\calm(B))$, see \cite[Section 5]{thomsen},
then similar as above we achieve
\begin{eqnarray*}
\mbox{Ad}(w) \circ (\varphi \hat {\otimes} id_{\calk_G})
&\in& \mbox{Hom}_G(A \otimes^X \calk_G,\calm(B \otimes^X \calk_G)).
\end{eqnarray*}
In other words, if $(\hat \varphi,u)$ is degenerate, so is $\Theta(\hat \varphi,u)$.

Our next aim is to show that $\Theta$ respects unitary equivalence.
Assume that two twisted extensions $(\varphi,u)$ and $(\psi,v)$ are unitarily equivalent,
$(\varphi,u) \cong(\psi,v)$, see \cite[Section 5]{thomsen}.
That means, that there exists a unitary $s \in \calm(B)$ such that
\begin{equation}   \label{comp2}
\mbox{Ad} \, q_B(s) \circ \varphi = \psi , \qquad v_g \overline \beta_g(s) - s u_g \in B
\end{equation}
for all $g \in G$. Let $t \in \calm(B \otimes^X \calk_G)$ denote
the unitary $w$ for the twisted extension $(\psi,v)$ from Lemma \ref{lemma62}.
From 
(\ref{comp2}) it follows
\begin{eqnarray}
&& \mbox{Ad} \, q_{B \otimes^X \calk_G} (t (s \otimes 1_{\calk_G}) w^*)
\circ \mbox{Ad} q_{B \otimes^X \calk_G} (w)
\circ (\varphi \hat \otimes id_{\calk_G})    \nonumber \\
&=&
\mbox{Ad} q_{B \otimes^X \calk_G} (t)
\circ (\psi \hat \otimes id_{\calk_G}).    \label{comp6}
\end{eqnarray}
Moreover, computing $\overline{\beta_g \otimes \rho_g}(w^*) = (\beta_{g g^{-1}} \otimes \rho_{g g^{-1}}) (u_g^* \otimes 1)
w^*$
by multiplying (\ref{comp3}) from the right with
${\beta_{g^{-1}} \otimes \rho_{g^{-1}}}$ and taking the adjoint,
and similarly computing $\overline{\beta_g \otimes \rho_g}(t) = t  (v_g \otimes 1)
(\beta_{g g^{-1}} \otimes \rho_{g g^{-1}})$,
we have for all $m \in \calm(B)$ and $k \in \calk_G$
(in $\calm(B \otimes^X \calk_G)$)
\begin{eqnarray*}
&& \Big ( \overline{\beta_g \otimes \rho_g} \big( t (s \otimes 1) w^* \big )
- \overline{\beta_{g g^{-1}} \otimes \rho_{g g^{-1}}} \big ( t (s \otimes 1) w^* \big ) \Big) w (m \otimes k) w^* \\
&=& \Big (
t (v_g \otimes 1)
 \overline{\beta_g \otimes \rho_g} \big (s \otimes 1 \big ) (u_g^* \otimes 1) w^*
-  t \big(\overline \beta_{g g^{-1}}(s) \otimes \overline \rho_{g g^{-1}}(1) \big) w^*  \Big) w (m \otimes k) w^*  \\
&=&
t \Big( \big(v_g
 \overline\beta_g(s) u_g^* m - \overline \beta_{g g^{-1}}(s) m \big ) \otimes \overline \rho_g(1) k \Big ) w^*
\qquad \in \qquad B \otimes^X \calk_G
\end{eqnarray*}
by the fact that $v_g \overline \beta_g(s) u_g^* - \beta_{g g^{-1}}(s) \in B$ by
(\ref{comp2}) and (\ref{remarkcocycle}).

The last computation shows that
\begin{eqnarray}
&& \widehat{\beta_g \otimes \rho_g} \Big( q_{B \otimes^{X} \calk_G} \big (t (s \otimes 1) w^* \big ) \Big )
\, \mbox{Ad} \, q_{B \otimes^X \calk_G}(w) \circ (\varphi \hat \otimes id_{\calk_G})    \nonumber  \\
&=& \widehat{\beta_{g g^{-1}} \otimes \rho_{g g^{-1}}} \Big( q_{B \otimes^X \calk_G} \big (t (s \otimes 1) w^* \big ) \Big )
\, \mbox{Ad} \, q_{B \otimes^X \calk_G}(w) \circ (\varphi \hat \otimes id_{\calk_G}).    \label{comp5}
\end{eqnarray}
An application of Lemma \ref{lemma61} to identities (\ref{comp6}) and (\ref{comp5})
(recall (\ref{comp7}))
shows that $\Theta(\varphi,u) + 0$ (sum operator, cf. \cite[Lemma 3.1]{thomsen}) and $0+\Theta(\psi,v)$
are unitarily equivalent in the sense of Definition \ref{defUnitarilyEqu}.

We leave the verification of the last claim, that $\Theta$ respects direct sums, and so also invertible elements go to
invertible elements, to the reader by verbatim following the corresponding proof
in the last part of \cite[Section 6]{thomsen}.
\end{proof}

\section{Equivariant $KK$-theory and $C^*$-extensions}

\label{section7}

The aim of 
\cite[Section 7]{thomsen}
is the proof of the main result, the identification of
extension groups with Kasparov groups, see Theorems \ref{theorem71}
and \ref{theorem72}. 

Again, in this section we must replace the tensor products $A \otimes \calk_G$ and so on by the balanced tensor
products $A \otimes^X \calk_G$ and so forth. 
For the discussion in \cite[Section 7.2]{thomsen} we remark that $\calk_G$ is equivalent to $C_0(X)$ in $KK^G$.
This was proved in \cite[Proposition 5.14]{burgiBCtriangulated}. (Note that $\calk(L^2(G))$ is equivalent
to $\calk(L^2(G) \oplus C_0(X))$ which is equivalent to $\calk(C_0(X))$ in $KK^G$.)

We begin by restating the main results of Thomsen's paper, slightly adapted to our setting.
These are, of course, also the main results of this note.
Recall from \cite[Section 4]{thomsen} that $KK^1_G(A,B):= KK^G(A, B \oplus B)$, where $A$ and $B$ are trivially graded and $B \oplus B$
is obviously graded by the flip operator.

\begin{theorem}[Cf. Theorem 7.1 of \cite{thomsen}]    \label{theorem71}
Assume that $G$ is a countable, $E$-continuous inverse semigroup.
We may identify the $KK_G^1$-groups with the extension groups of the stabilized $G$-algebras.
More precisely, for all ungraded separable $G$-algebras $A$ and $B$, where $B$ is stable, we have
\begin{eqnarray*}
&& KK_G^1(A,B)
\;\cong \;
\mbox{Ext}_G(A \otimes^X \calk_G, B \otimes^X \calk_G) \; \cong \;
\mbox{Ext}_G^h(A \otimes^X \calk_G, B \otimes^X \calk_G)  \\
&\cong & \mbox{Ext}_{G,t}(A, B)  \; \cong \; \widetilde{KK}_G^1(A,B).
\end{eqnarray*}
\end{theorem}

The last two isomorphisms are a restatement of \cite[Theorem 4.3]{thomsen} and
\cite[Theorem 5.3]{thomsen}. 
Of course, the first two isomorphisms hold also for unstable $B$.
The following theorem just states that the second isomorphism is canonical.

\begin{theorem}[Cf. Theorem 7.2 of \cite{thomsen}]    \label{theorem72}
Assume that $G$ is $E$-continuous.
Let $\varphi,\psi \in \mbox{Hom}_G \big(A \otimes^X \calk_G, Q(B \otimes^X \calk_G) \big)$
be invertible $G$-extensions.
Then $[\varphi]=[\psi]$ in $\mbox{Ext}_G(A \otimes^X \calk_G, B \otimes^X \calk_G)$
if and only if $\varphi$ and $\psi$ are homotopic, that is, there is an invertible
extension $\Phi \in \mbox{Hom}_G \big(A \otimes^X \calk_G, Q( I B \otimes^X \calk_G) \big)$
such that $\widehat{\pi_0} \circ \Phi = \varphi$ and $\widehat{\pi_1} \circ \Phi = \psi$.
\end{theorem}

\begin{lemma} 
\cite[Lemma 7.3]{thomsen} holds verbatim. (Recall that an operator $u \in \calm((A,\alpha))$
is {\em $\overline \alpha$-invariant}
if $\alpha_g \circ u = u \circ \alpha_g$ for all $g \in G$.)
\end{lemma}

\begin{proof}
In \cite[Lemma 7.3]{thomsen} we have to alter two lines in its proof. We have
to take
\begin{eqnarray*}
\widehat{\beta_g \otimes id_{M_2}}  \left ( \left (
\begin{matrix}
\overline \psi (U) & 0 \\
0 & \overline \psi (U^*)
\end{matrix}
\right ) \right ) x
&=&
\widehat{\beta_{g g^{-1}} \otimes id_{M_2}}  \left ( \left (
\begin{matrix}
\overline \psi (U) & 0 \\
0 & \overline \psi (U^*)
\end{matrix}
\right ) \right ) x
\end{eqnarray*}
and
\begin{eqnarray*}
\widehat{\beta_g \otimes id_{M_2}}  \big(q_{B \otimes M_2}(W) \big) \left (
\begin{matrix}
\overline \psi (a) & 0 \\
0 &  0
\end{matrix}
\right )
&=&
\widehat{\beta_{g g{-1}} \otimes id_{M_2}}  \big(q_{B \otimes M_2}(W) \big) \left (
\begin{matrix}
\overline \psi (a) & 0 \\
0 & 0
\end{matrix}
\right )
\end{eqnarray*}
instead of the corresponding identities without appearance of $\widehat{\beta_{g g{-1}} \otimes id_{M_2}}$
in Thomsen's paper.
Of course, instead of \cite[Lemma 6.1]{thomsen} we have to apply Lemma \ref{lemma61}
in the proof.
\end{proof}

Most of the remainder of Section 7 of Thomsen's paper after \cite[Lemma 7.3]{thomsen}
is dedicated to the following lemma stated at the beginning of \cite[Section 7.3]{thomsen}.

\begin{lemma}
There exists a well-defined, injective map
\begin{eqnarray*}
&& \kappa:\mbox{Hom}_G \big(A \otimes^X \calk_G, Q(B \otimes^X \calk_G) \big)/\sim  \\
&&\quad \rightarrow \quad
\mbox{Hom}_G \big (A \otimes^X \calk_G \otimes^X \calk_G, Q(B \otimes^X \calk_G \otimes^X \calk_G) \big)/\sim
\end{eqnarray*}
given by $[\psi] \mapsto [\psi \hat \otimes id_{\calk_G}]$,
where $\sim$ denotes the equivalence relation on $\mbox{Hom}_G(A,Q(B))$ whose
quotient defines the set of invertible elements called $\mbox{Ext}_G(A,B)$.
\end{lemma}

\begin{proof}
As in Thomsen \cite{thomsen} we remark
that $\kappa$ is similarly defined like $\Theta$ of Lemma \ref{afterlemma62} and so the proof
that $\kappa$ is well-defined is similar as the corresponding proof there by an application of Lemma \ref{lemma61}.

After \cite[Lemma 7.3]{thomsen} Thomsen considers the $G$-equivariant Hilbert space $\C \oplus L^2(G)$
which in our case turns to the $G$-Hilbert $C_0(X)$-module
$C_0(X) \oplus L^2(G)$ with diagonal $G$-action $\tau_g^+(\lambda,\psi) := g(\lambda) \oplus g(\psi)$
for all $g \in G, \lambda \in C_0(X)$ and $\psi \in L^2(G)$.
This induces a canonical $G$-action 
on the $C^*$-algebra 
$$\calk_G^+  \,:= \, \calk \big(C_0(X)^\infty \big) \otimes^X \calk \big(C_0(X) \oplus L^2(G) \big).$$
We proceed then as in Thomsen's paper after \cite[Lemma 7.3]{thomsen}. Thomsen considers
a similar map, $\kappa^+$ say, which replaces the new copies of $\calk_G$ in the image
of $\kappa$ by $\calk_G^+$ and maps $[\psi]$ to $[\psi \hat \otimes \calk_G^+]$.
A $G$-invariant projection $e \in \calk_G^+$ is chosen for which we take $e:=e_{11} \otimes e_{11} \in \calk_G^+$  
where $e_{11}=\theta_{1_{C_0(X)},1_{C_0(X)}} \in \calk(C_0(X) \oplus L^2(G)), \calk(C_0(X)^\infty)$ denotes the corner
projection.
%
By Corollary \ref{corollarymingophillips26} we may choose a $G$-invariant isometry $V \in \calm(B \otimes^X \calk_G
\otimes^X \calk_G^+)$ such that $V V^* = 1_{\calm(B \otimes^X \calk_G)} \otimes e$.
We can then literally proceed as in Thomsen's paper to show that $\kappa^+$ is injective.
Also the remaining proof where it is shown that $\kappa$ is injective
goes verbatim through in our setting (with only obvious adaption as taking $A \otimes^X \calk_G$ rather than $A \otimes \calk_G$).
A certain $G$-invariant isometry $W \in \calm(A \otimes^X \calk_G \otimes^X \calk_G^+)$
with range projection $1_{\calm(A \otimes^X \calk_G \otimes^X \calk_G)}$
is then again chosen by Corollary \ref{corollarymingophillips26}.
\end{proof}


\section{Adaption of a paper of Mingo and Phillips}

\label{sectionMingoPhillips}

In this section we introduce the {\em compatible} $\ell^2(G)$-space $\widehat \ell^2(G)$ for an inverse semigroup
$G$
which replaces the corresponding classical space for groups with its regular representation.
This new notion is necessary for $\widehat \ell^2(G)$ to become a $G$-Hilbert $C_0(X)$-module, and so particularly
$\call(\ell^2(G))$ a $G$-algebra. The classical $\ell^2(G)$-space would not become a $G$-Hilbert $\C$-module.
Moreover, we restate some central results from the paper of Mingo and Phillips \cite{mingophillips}
for this new $\ell^2$-space. All from this section is actually taken from \cite{burgiBCtriangulated}.

\begin{definition}   \label{defC0X}
{\rm
Define $X$ to be the locally compact Hausdorff space and Gelfand spectrum of the commutative
$C^*$-algebra $C^*(E)$ which is freely generated by the commuting, self-adjoint projections $e \in E$.
That is, $C_0(X) \cong C^*(E)$ via $1_e \leftrightarrow e$ for all $e \in E$
(so we identify in our notation the projection
$e$ with its carrier set in $X$).
We turn this $C^*$-algebra to a $G$-algebra under the $G$-action
$g(1_e) := 1_{g e g^{-1}}$ for all $g \in G$ and $e \in E$.


}
\end{definition}

In the next few paragraphs (until Lemma \ref{lemmalinindepen}) 
we shall identify elements $e \in E$ with its characteristic function $1_e$ in $C_0(X)$.
Write $\mbox{Alg}^*(E)$ for the dense $*$-subalgebra of $C_0(X)$ generated by the characteristic functions
$1_e$ for all $e \in E$.
Moreover,
write $\bigvee_i f_i :X \rightarrow \C$ for the pointwise supremum of a family of functions $f_i: X \rightarrow \C$.
We shall use the well-known order relation on $G$ defined by $g \le h$ iff $g = eh$ for some $e \in E$.

\begin{definition}   \label{definitionEcontinuous}
{\rm
An inverse semigroup $G$ is called {\em $E$-continuous} if
the function $\bigvee \{ e \in E |\, e \le g\} \in \C^X$ (in precise notation: $\bigvee \{1_e \in C_0(X) |\,e\in E,\, e \le g\} \in \C^X$) 
is a {\em continuous} function in $C_0(X)$
for all $g \in G$.
}
\end{definition}

\begin{lemma}   \label{lemmaEcontinuity}
An inverse semigroup $G$ is {$E$-continuous} if and only if for every $g \in G$ there exists a finite
subset $F \subseteq E$ such that
$\bigvee \{e \in E |\, e \le g\} =\bigvee \{e \in F |\, e \le g\}$.
\end{lemma}

\begin{proof}
If $\bigvee \{e \in E |\, e \le g\} = 1_K \in C_0(X)$ for a clopen subset $K \subseteq X$ then $K$ must be compact. Hence $K = \bigcup \{\,\mbox{carrier}(1_e) \subseteq X\,| \, e \in E ,\, e \le g\}$ allows a finite subcovering.
\end{proof}

\begin{definition}[Compatible $L^2(G)$-space]     \label{defCompatibleL2}
{\rm
Let $G$ be an $E$-continuous inverse semigroup.
Write $c$ for the linear span of all functions $\varphi_g: G \rightarrow \C$ (in the linear space $\C^G$) defined by
\begin{eqnarray*}
\varphi_g(t) & :=& 1_{\{t \le g\}} 
\end{eqnarray*}
(characteristic function)
for all $g,t \in G$.
Endow $c$ with the $G$-action $g(\varphi_h) := \varphi_{g h}$ for all $g,h \in G$.
Turn $c$ to an $\mbox{Alg}^*(E)$-module by setting $\xi e := e(\xi)$ for all $\xi \in c$ and $e \in E$.
Define an $\mbox{Alg}^*(E)$-valued inner product on $c$ by
\begin{eqnarray}  \label{identinnerprodphi}
\langle \varphi_g, \varphi_h \rangle &:=& \bigvee \{e \in E \,|\, eg = e h ,
\,e \le g g^{-1} h h^{-1} \}.
\end{eqnarray}
The norm completion of $c$ is a 
$G$-Hilbert $C_0(X)$-module denoted by $\widehat{\ell^2}(G)$.
}
\end{definition}

We discuss the last definition.
At first notice that
$\langle \varphi_g , \varphi_h\rangle = g g^{-1} \bigvee \{e \in E|\,e = e h g^{-1}\}$
(observe that $e = e h g^{-1}$ implies $e \le h g^{-1} g h^{-1}$),
so that by $E$-continuity $\langle \varphi_g , \varphi_h\rangle$ is in $C_0(X)$ and actually even in $\mbox{Alg}^*(E)$ by
Lemma \ref{lemmaEcontinuity}, and $e \in E$ in (\ref{identinnerprodphi}) can be replaced by $e \in F$
for some finite subset $F \subseteq E$.
The identities
$\langle \varphi_g , \varphi_h \rangle = \langle \varphi_h , \varphi_g \rangle$, $\langle \varphi_g , \varphi_h f \rangle
= \langle \varphi_g f , \varphi_h \rangle
= \langle \varphi_g , \varphi_h \rangle f$,
$j(\langle \varphi_g , \varphi_h \rangle) = \langle j(\varphi_g) , j(\varphi_h) \rangle$
for all $g,h,j \in G$ and $f \in E$ are easy to check.
We note that (\ref{identinnerprodphi}) is positive definite. Indeed, assume
$\langle x,x\rangle = 0$ for $x=\sum_{i=1}^n \lambda_i \varphi_{g_i}$
with nonzero $\lambda_i \in \C$ and $g_i \in G$ mutually different.
Choose $g_j$ such that no other $g_i$ satisfies $g_j g_j^{-1} < g_i g_i^{-1}$.
Hence, $\langle \varphi_{g_j},\varphi_{g_j}\rangle = g_j g_j^{-1}$ but 
$\langle \varphi_{g_i},\varphi_{g_k}\rangle \neq g_j g_j^{-1}$ for all combinations where $i \neq k$.
By linear independence of the projections $E$ in $\mbox{Alg}^*(E)$ $\lambda_j$ must be zero;
contradiction.
The last proof also shows the following lemma.

\begin{lemma}    \label{lemmalinindepen}
The vectors $(\varphi_g)_{g \in G} \subseteq \widehat{\ell^2}(G)$ are linearly independent.
\end{lemma}



Every $G$-algebra $(A,\alpha)$ is endowed with a $*$-homomorphism $\pi:C_0(X) \cong C^*(E) \rightarrow Z\calm(A)$ given by $\pi(e)(a)= \alpha_e(a)$ for all $e \in E$ and $a \in A$.
If $G$ has a unit than $A$ is a $C_0(X)$-algebra in the sense of Kasparov \cite[Definition 1.5]{kasparov1988}; but this extra property
is unimportant in this note.

\begin{definition}   \label{defBalancedTensor}
{\rm
We define (cf. Kasparov \cite[Definition 1.6]{kasparov1988}) the balanced tensor product $A \otimes^X B$ of two $G$-algebras $A$ and $B$ as the quotient
of $A \otimes_{max} B$ divided by the relations $e(a) \otimes b \equiv a \otimes e(b)$ for all $e \in E, a \in A$ and $b \in B$.
}
\end{definition}

\begin{definition}  \label{defL2hilbertmodule}
{\rm
Let $\cale$ be a $G$-Hilbert $B$-module. Then
$L^2 (G,\cale) := \widehat{\ell^2}(G) \otimes^X \cale$ is a $G$-Hilbert $B$-module, where 
$\otimes^X$ denotes the $C_0(X)$-balanced exterior tensor product as defined by Le Gall \cite[Definition 4.2]{legall1999}
(or in this case equivalently, the internal tensor product $\otimes_{C_0(X)}$).
}
\end{definition}

\begin{lemma}[Cf. Lemma 2.3 of \cite{mingophillips}]    \label{lemmamingophillips23}
If $\cale_1$ and $\cale_2$ are $G$-Hilbert $A$-modules which are isomorphic as Hilbert $A$-modules then $L^2(G,\cale_1)$
and $L^2(G,\cale_2)$ are isomorphic as $G$-Hilbert $A$-modules.
\end{lemma}

\begin{proof}
Let $u \in \call(\cale_1,\cale_2)$ be a unitary operator. Note that $g g^{-1} \in \call(\cale_i)$ commutes
with $u$ for all $g \in G$ since $u$ is $A$-linear and $g g^{-1}(\xi) a = \xi g g^{-1}(a)$ for all $\xi \in \cale_i, a \in A$.
Then it can be checked that $V: L^2(G,\cale_1) \rightarrow  L^2(G,\cale_2)$
given by $V(\varphi_g \otimes \xi) := \varphi_g \otimes g u g^{-1}(\xi)$
defines an isomorphism of $G$-Hilbert $A$-modules. 
We show that $V$ is $G$-equivariant. 
For $h \in G$ we have
\begin{eqnarray*}
h\big(V(\varphi_g \otimes \xi )\big ) &=& h \varphi_{g} \otimes h g u g^{-1} h^{-1} h (\xi) \\
&=& \varphi_{h g} \otimes hg(u) (h(\xi)) \\
&=& V\big (h(\varphi_g \otimes \xi) \big ),
\end{eqnarray*}
because $h^{-1} h \in \call(\cale_i)$ commutes with $g u g^{-1} \in \call(\cale_1,\cale_2)$.

For the inner product
note that $\langle \varphi_g , \varphi_h \rangle = \sum_{f \in F} f$ for a finite set $F \subseteq E$
with $fg = f h$ and $f \le g g^{-1} h h^{-1}$ by Lemma \ref{lemmaEcontinuity}, so that
\begin{eqnarray*}
\langle V(\varphi_g \otimes \xi),V(\varphi_h \otimes \eta)\rangle
&=& \sum_{f \in F} f \otimes \langle fgug^{-1} f(\xi),f h u h^{-1} f(\eta) \rangle\\
&=& \langle \varphi_g \otimes \xi,\varphi_h \otimes \eta\rangle.
\end{eqnarray*}
\end{proof}

\begin{corollary}[Cf. Theorem 2.4 of \cite{mingophillips}]
Let $\cale$ be a $G$-Hilbert $A$-module which is countably generated and full as a Hilbert $A$-module.
Then $L^2(G,\cale)^\infty$ is isomorphic to $L^2(G,A)^\infty$ by a $G$-equivariant isomorphism
of Hilbert $A$-modules.
\end{corollary}

\begin{proof}
Same proof as in Mingo and Phillips \cite{mingophillips}, Theorem 2.4, but by applying
Lemma \ref{lemmamingophillips23} instead of \cite[Lemma 2.3]{mingophillips}.
\end{proof}

\begin{corollary}[Cf. Corollary 2.6 of \cite{mingophillips}]    \label{corollarymingophillips26}
Let $(A,\alpha)$ be a $G$-algebra and suppose that $A$ has a strictly positive element.
If $p \in \calm(A)$ is a full $G$-invariant projection then $p \otimes 1_{\calm(\calk_G)} \sim 
1_{\calm(A)} \otimes
1_{\calm(\calk_G)}$
(Murray--von Neumann equivalence) in $\calm(A \otimes^X \calk_{G})$
by a $G$-invariant partial isometry.
\end{corollary}

\begin{proof}
The proof of the original goes verbatim through. The usage of the balanced tensor product $\otimes^X$
instead of $\otimes$ is obligatory.
\end{proof}

\bibliographystyle{plain}
\bibliography{references}

\begin{thebibliography}{10}

\bibitem{burgiBCtriangulated}
B.~Burgstaller.
\newblock {Attempts to define a Baum--Connes map via localization of categories
  for inverse semigroups}.
\newblock {preprint arXiv:1506.08412}.

\bibitem{burgiUniversalKK}
B.~Burgstaller.
\newblock {The universal property of inverse semigroup equivariant
  $KK$-theory}.
\newblock {preprint arXiv:1405.1613}.

\bibitem{burgiSemimultiKK}
B.~Burgstaller.
\newblock {Equivariant $KK$-theory for semimultiplicative sets}.
\newblock {\em New York J. Math.}, 15:505--531, 2009.

\bibitem{cuntz1998}
J.~{Cuntz}.
\newblock {A general construction of bivariant $K$-theories on the category of
  $C^*$-algebras.}
\newblock In {\em {Operator algebras and operator theory. Proceedings of the
  international conference, Shanghai, China, July 4--9, 1997}}, pages 31--43.
  Providence, RI: American Mathematical Society, 1998.

\bibitem{kasparov1980}
G.G. Kasparov.
\newblock {Hilbert C*-modules: Theorems of Stinespring and Voiculescu.}
\newblock {\em J. Oper. Theory}, 4:133--150, 1980.

\bibitem{kasparov1981}
G.G. Kasparov.
\newblock {The operator K-functor and extensions of C*-algebras.}
\newblock {\em Math. USSR, Izv.}, 16:513--572, 1981.

\bibitem{kasparov1988}
G.G. Kasparov.
\newblock {Equivariant KK-theory and the Novikov conjecture.}
\newblock {\em Invent. Math.}, 91(1):147--201, 1988.

\bibitem{legall1999}
P.-Y. Le~Gall.
\newblock {Equivariant Kasparov theory and groupoids. I. (Th\'eorie de Kasparov
  \'equivariante et groupo\"{\i}des. I.)}.
\newblock {\em K-Theory}, 16(4):361--390, 1999.

\bibitem{mingophillips}
J.A. Mingo and W.J. Phillips.
\newblock {Equivariant triviality theorems for Hilbert $C\sp*$-modules.}
\newblock {\em Proc. Am. Math. Soc.}, 91:225--230, 1984.

\bibitem{thomsen2}
K.~{Thomsen}.
\newblock {The universal property of equivariant $KK$-theory.}
\newblock {\em {J. Reine Angew. Math.}}, 504:55--71, 1998.

\bibitem{thomsen}
K.~{Thomsen}.
\newblock {Equivariant $KK$-theory and $C^*$-extensions.}
\newblock {\em {$K$-Theory}}, 19(3):219--249, 2000.

\end{thebibliography}

\end{document}